\documentclass[12pt, a4paper]{article}
\usepackage{amsmath, amsthm, amssymb}
\usepackage{mathrsfs}
\usepackage{bm}
\usepackage[utf8]{inputenc}
\usepackage[T1]{fontenc}
\usepackage{lmodern}
\usepackage{microtype}
\usepackage{geometry}
\usepackage{hyperref}
\hypersetup{
    colorlinks=true,
    linkcolor=blue,
    filecolor=magenta,      
    urlcolor=cyan}
\newtheorem{theorem}{Theorem}[section]

\newtheorem{corollary}[theorem]{Corollary}
\theoremstyle{definition}
\newtheorem{definition}[theorem]{Definition}

\numberwithin{equation}{section}

\newcommand{\B}{\mathbb{B}}
\newcommand{\E}{\mathbb{E}}

\title{On Apostol-Type Mersenne-Bernoulli and Mersenne-Euler Polynomials}

\author{%
  Artatrana Suna\textsuperscript{a},
  Prasanta Kumar Ray\textsuperscript{b}\textsuperscript{*}%
}

\date{}

\begin{document}

\maketitle

\vspace{-2.5em}

\begin{center}
  \textit{Department of Mathematics, Sambalpur University, India}
\end{center}

\begin{abstract}
In this paper, we introduce the Apostol-type Mersenne-Bernoulli and Mersenne-Euler polynomials of order $\alpha$. By employing the $M$-calculus, based on the Mersenne numbers, we establish explicit series representations, addition theorems, difference equations and convolution identities.
\end{abstract}



\begingroup
  \renewcommand{\thefootnote}{\alph{footnote}}
  \footnotetext[1]{Email: \texttt{suna.19972702@gmail.com}}
  \footnotetext[2]{Email: \texttt{prasantamath@suniv.ac.in} (Corresponding author)}
\endgroup

\section{Introduction}
Special polynomials have long played a central role in various branches of mathematics, including number theory, combinatorics, approximation theory, and mathematical physics. Among these, the classical Bernoulli polynomials $\{B_n(x)\}$ and Euler polynomials $\{E_n(x)\}$ occupy important positions. These are defined respectively by their exponential generating functions \cite{abramowitz1974handbook}
\[
\frac{t e^{tx}}{e^t-1}=\sum_{n=0}^\infty B_n(x)\frac{t^n}{n!},\quad |t|<2\pi,
\]
\[
\frac{2e^{tx}}{e^t+1}=\sum_{n=0}^\infty E_n(x)\frac{t^n}{n!},\quad |t|<\pi.
\]
These polynomials exhibit numerous properties, such as addition theorems, recurrence relations and differential equations, and are applied in diverse fields including numerical analysis, asymptotic expansions, and the study of zeta functions.

A significant generalization was introduced by Apostol \cite{apostol1951lerch}, who defined the Apostol-Bernoulli polynomials by incorporating a parameter $\lambda$. Subsequently, Luo and Srivastava \cite{luo2005some} systematically studied the Apostol-Bernoulli and Apostol-Euler polynomials of order $\alpha$, establishing fundamental properties and connections with the Hurwitz-Lerch zeta function. Many interesting works on Apostol-type polynomials and their various generalizations can be found in \cite{luo2005some, luo2006some, ozarslan2011unified}.

The study of special polynomials via integral representations and generating functions \cite{Cesarano2015, Cesarano2019}, degenerate and hypergeometric versions  \cite{BedoyaCes2023, CesaranoRam2022, CesRam2023, CesQuint2024, QuintRam2024, RamCeserano2022}, and the monomiality principle for orthogonal and Hermite-type polynomials \cite{CesQuintRam2024, dattoliCes2001, dattoli2001, RamCes2024, RamAlej2024} have further enriched the theory of special polynomials.

Over the past few years, the study of sequence-based calculus has emerged as a dynamic research area. Krot \cite{krot2004introduction} introduced the $F$-calculus based on Fibonacci numbers, which was further developed as Golden calculus \cite{pashaev2012golden}. This framework has been extensively applied to study Fibonacci-based analogues of classical special polynomials. Kus et al. \cite{kus2019bernoulli} introduced Bernoulli $F$-polynomials and Fibo-Bernoulli matrices. The authors in \cite{gulal2023apostol} defined Apostol Bernoulli-Fibonacci and Apostol Euler-Fibonacci polynomials whereas investigation of parametric forms of these polynomials via Golden calculus was done in \cite{kizilates2023parametric}. Most recently, Di\c{s}kaya \cite{diskaya2025apostol} introduced Apostol-type polynomials based on Padovan sequences, showing that the methodology extends naturally to other recursive sequences. For further studies along these lines, we refer the reader to \cite{duran2025FHermite, mersin2025apostol, Ramirez2024matrix, urieles2024frobenius}. 

Building on this rich literature, we introduce $M$-calculus based on Mersenne numbers, which are of the form $M_n = 2^n-1,\quad n \geq 0$. These numbers play a crucial role in number theory, particularly in the study of perfect numbers and large primes. In this work, we define and study the Mersenne version of Apostol-type Bernoulli and Euler polynomials through the $M$-calculus.

\section{Preliminaries}

In this section, we introduce the basic notions of $M$-calculus, which are built upon the Mersenne numbers. While these concepts are similar to their classical analogs, they will be essential for our subsequent study.
\begin{definition}[$M$-factorial]
The $M$-factorial is defined recursively by $M_0! = 1$ and $M_n! = M_n M_{n-1}!$ for $n\ge 1$.
\end{definition}

\begin{definition}[$M$-binomial coefficients]
For $0\le k\le n$, the $M$-binomial coefficient is defined by
\[
\binom{n}{k}_M = \frac{M_n!}{M_k!\,M_{n-k}!},
\]
with $\binom{n}{k}_M=0$ for $k>n$. 
\end{definition}
These coefficients satisfy the symmetry $\binom{n}{k}_M = \binom{n}{n-k}_M$, the boundary conditions $\binom{n}{0}_M=1$, $\binom{n}{1}_M=M_n$, and the recurrence
\[
\binom{n}{k}_M = 2^k\binom{n-1}{k}_M + \binom{n-1}{k-1}_M.
\]
Hence, the $M$-binomial coefficients are always integers.
\begin{definition}[$M$-exponential function]
The $M$-exponential function is defined by the power series
\[
e_M^{xt} = \begin{cases}
\displaystyle \sum_{n=0}^\infty x^n \frac{t^n}{M_n!}, & \text{if } x \neq 0,\\[1ex]
1, & \text{otherwise}.
\end{cases}
\]
\end{definition}
Unlike the classical one, the $M$-exponential function is not multiplicative rather we have $e_M^{xt}e_M^{yt}=e_M^{(x+_M y)t}.$ Here $+_M$ denotes the $M$-binomial addition defined by the $M$-binomial theorem:
\[
(x+_M y)^n = \sum_{k=0}^n \binom{n}{k}_M x^k y^{n-k}.
\]
We can define the Mersenne version of derivative and integral operator by the following manner.
\begin{definition}[$M$-derivative]
The $M$-derivative operator $\mathcal{D}^x$ is defined by its action on real or complex valued functions
\[
\mathcal{D}^x(f(x)) = \frac{f(2x)-f(x)}{x},\qquad x\neq 0,
\]
with $\mathcal{D}^x(f(0)) = \lim_{x\to 0}\mathcal{D}^x(f(x))$ when the limit exists.
\end{definition}
\noindent For monomials, we have
\[
\mathcal{D}^x(x^n) = M_n x^{n-1},\qquad n\ge 1.
\]
The $M$-derivative is linear and satisfies the product rule
\[
\mathcal{D}^x(f(x)g(x)) = g(2x)\mathcal{D}^x(f(x)) + f(x)\mathcal{D}^x(g(x)).
\]

\begin{definition}[$M$-integral]
For a function $f(x)$ with $\mathcal{D}^xF(x) = f(x),$ its $M$-integral is defined as
\[
\int_a^b f(x)\,d_M x = F(b)-F(a).
\]
\end{definition}
\noindent As an example we see that
\[
\int_0^1 x^{n}\,d_M x = \frac{1}{M_{n+1}}
\]
for $n\geq 0.$
\noindent The linearity of the $M$-integral follows directly from the definition of the $M$-derivative. The integration by parts formula can be obtained as
\[
\int_{a}^{b} f(x)\mathcal{D}^{x} (g(x))\,d_M x = f(b)g(b)-f(a)g(a)-\int_{a}^{b} \mathcal{D}^{x} (f(x)) g(2x)\,d_M x,
\]
\[
\int_{a}^{b} g(2x)\mathcal{D}^{x} (f(x)) \,d_M x = f(b)g(b)-f(a)g(a)-\int_{a}^{b} f(x)\mathcal{D}^{x} (g(x))\,d_M x.
\]
The $M$-calculus introduced above provides a complete algebraic framework that is analogous to the classical calculus while being intrinsically linked to the Mersenne numbers. The $M$-factorial, $M$-binomial coefficients, $M$-exponential, $M$-derivative, and $M$-integral satisfy natural identities that will serve as the foundation for defining and studying Apostol-type Mersenne–Bernoulli and Mersenne–Euler polynomials.

\section{Apostol-Mersenne-Bernoulli and Apostol-Mersenne-Euler Polynomials}

In this section, we introduce the Apostol-type generalizations of the Mersenne-Bernoulli and Mersenne-Euler polynomials by incorporating the complex parameters $\lambda,\alpha,$ where $\alpha$ denotes the order. Further, we systematically study these polynomials, establishing their fundamental properties including explicit representations, addition theorems, difference equations, $M$-derivative formulas, and convolution identities. 

\begin{definition}
The Apostol-Mersenne-Bernoulli polynomials $\B_{n,M}^{(\alpha)}(x;\lambda)$ and Apostol-Mersenne-Euler polynomials $\E_{n,M}^{(\alpha)}(x;\lambda)$ of order $\alpha$ are defined by the following generating functions
\begin{equation}\label{eq:apostol-bm-def}
\left(\frac{t}{\lambda e_M^t-1}\right)^\alpha e_M^{xt} = \sum_{n=0}^\infty \B_{n,M}^{(\alpha)}(x;\lambda)\frac{t^n}{M_n!},
\end{equation}
\begin{equation}\label{eq:apostol-em-def}
\left(\frac{2}{\lambda e_M^t+1}\right)^\alpha e_M^{xt} = \sum_{n=0}^\infty \E_{n,M}^{(\alpha)}(x;\lambda)\frac{t^n}{M_n!}.
\end{equation}
The Apostol-Mersenne-Bernoulli numbers are $\B_{n,M}^{(\alpha)}(\lambda)=\B_{n,M}^{(\alpha)}(0;\lambda)$, whereas the Apostol-Mersenne-Euler numbers are $\E_{n,M}^{(\alpha)}(\lambda)=\E_{n,M}^{(\alpha)}(0;\lambda)$. 
\end{definition}

The following theorem gives the explicit representation of the Apostol-type polynomials in terms of the corresponding numbers.

\begin{theorem}\label{thm:explicit}
For $n\ge 0$, the Apostol-type Mersenne-Bernoulli and Mersenne-Euler polynomials satisfy
\begin{equation}\label{eq:apostol-bm-explicit}
\B_{n,M}^{(\alpha)}(x;\lambda) = \sum_{k=0}^n \binom{n}{k}_M \B_{k,M}^{(\alpha)}(\lambda) x^{n-k},
\end{equation}
\begin{equation}\label{eq:apostol-em-explicit}
\E_{n,M}^{(\alpha)}(x;\lambda) = \sum_{k=0}^n \binom{n}{k}_M \E_{k,M}^{(\alpha)}(\lambda) x^{n-k}.
\end{equation}
\end{theorem}

\begin{proof}
From the generating function \eqref{eq:apostol-bm-def}, we have
\begin{align*}
\sum_{n=0}^\infty \B_{n,M}^{(\alpha)}(x;\lambda)\frac{t^n}{M_n!} &= \left(\frac{t}{\lambda e_M^t-1}\right)^\alpha e_M^{xt} \\
&= \left(\sum_{k=0}^\infty \B_{k,M}^{(\alpha)}(\lambda)\frac{t^k}{M_k!}\right)\left(\sum_{m=0}^\infty x^m \frac{t^m}{M_m!}\right) \\
&= \sum_{n=0}^\infty \left(\sum_{k=0}^n \frac{\B_{k,M}^{(\alpha)}(\lambda) x^{n-k}}{M_k! M_{n-k}!}\right) M_n! \frac{t^n}{M_n!} \\
&= \sum_{n=0}^\infty \left(\sum_{k=0}^n \binom{n}{k}_M \B_{k,M}^{(\alpha)}(\lambda) x^{n-k}\right) \frac{t^n}{M_n!}.
\end{align*}
Comparing coefficients yields \eqref{eq:apostol-bm-explicit}. The proof for \eqref{eq:apostol-em-explicit} is analogous.
\end{proof}

The explicit representation shows that $\mathbb{B}_{n,M}^{(\alpha)}(x;\lambda)$ and   $\mathbb{E}_{n,M}^{(\alpha)}(x;\lambda)$ are monic polynomials of degree $n.$ The following theorem provides the addition formulas for the Apostol-type polynomials.

\begin{theorem}[Addition formula]\label{thm:addition}
For any $x$ and $y$
\begin{equation}\label{eq:apostol-bm-addition}
\B_{n,M}^{(\alpha)}(x+_M y;\lambda) = \sum_{k=0}^n \binom{n}{k}_M \B_{k,M}^{(\alpha)}(x;\lambda) y^{n-k},
\end{equation}
\begin{equation}\label{eq:apostol-em-addition}
\E_{n,M}^{(\alpha)}(x+_M y;\lambda) = \sum_{k=0}^n \binom{n}{k}_M \E_{k,M}^{(\alpha)}(x;\lambda) y^{n-k}.
\end{equation}
\end{theorem}

\begin{proof}
Using the generating function \eqref{eq:apostol-bm-def} and the addition property of the $M$-exponential,
\begin{align*}
\sum_{n=0}^\infty \B_{n,M}^{(\alpha)}(x+_M y;\lambda)\frac{t^n}{M_n!} &= \left(\frac{t}{\lambda e_M^t-1}\right)^\alpha e_M^{(x+_M y)t} \\
&= \left(\frac{t}{\lambda e_M^t-1}\right)^\alpha e_M^{xt} e_M^{yt} \\
&= \left(\sum_{k=0}^\infty \B_{k,M}^{(\alpha)}(x;\lambda)\frac{t^k}{M_k!}\right)\left(\sum_{m=0}^\infty y^m\frac{t^m}{M_m!}\right) \\
&= \sum_{n=0}^\infty \left(\sum_{k=0}^n \binom{n}{k}_M \B_{k,M}^{(\alpha)}(x;\lambda) y^{n-k}\right)\frac{t^n}{M_n!}.
\end{align*}
Comparing coefficients yields \eqref{eq:apostol-bm-addition}. We can establish \eqref{eq:apostol-em-addition} similarly.
\end{proof}

The next theorem is a difference equation which dictates the relationship between polynomials of order $\alpha$ and $\alpha-1$.

\begin{theorem}[Difference equation]\label{thm:difference}
For $n\ge 1$, the Apostol-type polynomials satisfy
\begin{equation}\label{eq:apostol-bm-difference}
\lambda \B_{n,M}^{(\alpha)}(x+_M1;\lambda) - \B_{n,M}^{(\alpha)}(x;\lambda) = M_n \B_{n-1,M}^{(\alpha-1)}(x;\lambda),
\end{equation}
\begin{equation}\label{eq:apostol-em-difference}
\lambda \E_{n,M}^{(\alpha)}(x+_M1;\lambda) + \E_{n,M}^{(\alpha)}(x;\lambda) = 2 \E_{n,M}^{(\alpha-1)}(x;\lambda).
\end{equation}
\end{theorem}

\begin{proof}
We first prove \eqref{eq:apostol-bm-difference}. Consider the identity
\begin{align*}
\left(\frac{t}{\lambda e_M^t-1}\right)^\alpha e_M^{(x+_M1)t} &= \left(\frac{t}{\lambda e_M^t-1}\right)^{\alpha-1} e_M^{xt} \frac{t e_M^t}{\lambda e_M^t-1}.
\end{align*}
Now note that
\begin{align*}
\frac{t e_M^t}{\lambda e_M^t-1} &= \frac{1}{\lambda}\left(t + \frac{t}{\lambda e_M^t-1}\right).
\end{align*}
Therefore,
\begin{align*}
\sum_{n=0}^\infty \lambda\B_{n,M}^{(\alpha)}(x+_M1;\lambda)\frac{t^n}{M_n!} &= t\left(\frac{t}{\lambda e_M^t-1}\right)^{\alpha-1} e_M^{xt} + \left(\frac{t}{\lambda e_M^t-1}\right)^{\alpha} e_M^{xt} \\
&= \sum_{n=1}^\infty M_n \B_{n-1,M}^{(\alpha-1)}(x;\lambda)\frac{t^n}{M_n!} + \sum_{n=0}^\infty \B_{n,M}^{(\alpha)}(x;\lambda)\frac{t^n}{M_n!}.
\end{align*}
Now, extracting the terms for $n \geq 1$ and equating the corresponding coefficients yield the required result for the Bernoulli case. The proof for \eqref{eq:apostol-em-difference} follows similarly using the identity $\frac{2e_M^t}{\lambda e_M^t+1} = \frac{2}{\lambda} - \frac{2}{\lambda(\lambda e_M^t+1)}$.
\end{proof}
\noindent Taking $\alpha = 1$ results in the following corollary.
\begin{corollary} \label{th:numcol}
If $n \geq 0$, then
\[
\lambda \B_{n+1,M}(x+_M1;\lambda) - \B_{n+1,M}(x;\lambda) = M_{n+1} x^{n},
\]
\[
\lambda \E_{n,M}(x+_M1;\lambda) + \E_{n,M}(x;\lambda) = 2 x^n.
\]
\end{corollary}

The following theorem provides recurrence relations for the Apostol-type Mersenne numbers.

\begin{theorem}\label{thm:number-recurrence}
For $n\ge 1$, the Apostol-type Mersenne-Bernoulli numbers satisfy
\begin{equation}\label{eq:apostol-bm-number-recurrence}
\lambda \sum_{k=0}^n \binom{n+1}{k}_M \B_{k,M}^{(\alpha)}(\lambda) = M_{n+1} \B_{n,M}^{(\alpha-1)}(\lambda) + \B_{n+1,M}^{(\alpha)}(\lambda),
\end{equation}
whereas the Apostol-type Mersenne-Euler numbers satisfy
\begin{equation}\label{eq:apostol-em-number-recurrence}
\lambda \sum_{k=0}^n \binom{n}{k}_M \E_{k,M}^{(\alpha)}(\lambda) + \E_{n,M}^{(\alpha)}(\lambda) = 2 \E_{n,M}^{(\alpha-1)}(\lambda).
\end{equation}
\end{theorem}

\begin{proof}
Setting $x=0$ in Theorem \ref{thm:difference} and using the addition formula with $x=0$, $y=1$ gives
\begin{align*}
\lambda \sum_{k=0}^n \binom{n}{k}_M \B_{k,M}^{(\alpha)}(\lambda) - \B_{n,M}^{(\alpha)}(\lambda) = M_n \B_{n-1,M}^{(\alpha-1)}(\lambda).
\end{align*}
Replacing $n$ by $n+1$ yields \eqref{eq:apostol-bm-number-recurrence}. The Euler case follows similarly from \eqref{eq:apostol-em-difference}.
\end{proof}

The next set of results establishes the relationship between these polynomials through an addition formula.
\begin{theorem}\label{thm:mixed-bernoulli}
For $n\geq 0$, the following mixed addition formula holds
\[
\B_{n,M}^{(\alpha)}(x+_M y;\lambda)
= \sum_{k=0}^n \binom{n}{k}_M
\Bigl[\B_{k,M}^{(\alpha)}(y;\lambda) + \frac{M_k}{2}\,\B_{k-1,M}^{(\alpha-1)}(y;\lambda)\Bigr]
\E_{n-k,M}(x;\lambda),
\]
where $\B_{-1,M}^{(\alpha-1)}(y;\lambda)=0$ and $\E_{n-k,M}(x;\lambda)=\E_{n-k,M}^{(1)}(x;\lambda)$.
\end{theorem}

\begin{proof}
By Corollary \ref{th:numcol}, we have 
\[
x^{n-k}= \frac{1}{2}\left(\E_{n-k,M}(x;\lambda)+\lambda\sum_{j=0}^{n-k}\binom{n-k}{j}_M\E_{j,M}(x;\lambda)\right)
\]
for $n \geq k \geq 0$. Inserting this into \eqref{eq:apostol-bm-addition} results in
\begin{align*}
\B_{n,M}^{(\alpha)}(x+_M y;\lambda)=
\frac{1}{2}&\sum_{k=0}^n\binom{n}{k}_M\B_{k,M}^{(\alpha)}(y;\lambda)\E_{n-k,M}(x;\lambda)\\
&+\frac{\lambda}{2}\sum_{k=0}^n\binom{n}{k}_M\B_{k,M}^{(\alpha)}(y;\lambda)
\sum_{j=0}^{n-k}\binom{n-k}{j}_M\E_{j,M}(x;\lambda).
\end{align*}
In the second double sum, interchange the order of summation using the $M$-binomial identity
\[
\binom{n}{k}_M\binom{n-k}{j}_M = \binom{n}{j}_M\binom{n-j}{k}_M
\]
to obtain
\begin{align*}
\B_{n,M}^{(\alpha)}(x+_M y;\lambda)=
\frac{1}{2}&\sum_{k=0}^n\binom{n}{k}_M\B_{k,M}^{(\alpha)}(y;\lambda)\E_{n-k,M}(x;\lambda)\\
&+\frac{\lambda}{2}\sum_{j=0}^n\binom{n}{j}_M\E_{j,M}(x;\lambda)
\sum_{k=0}^{n-j}\binom{n-j}{k}_M\B_{k,M}^{(\alpha)}(y;\lambda).
\end{align*}
The inner sum in the second sum is $\B_{n-j,M}^{(\alpha)}(x+_M1;\lambda)$. Thus,
\begin{equation}\label{eq:bm-em1}
\B_{n,M}^{(\alpha)}(x+_M y;\lambda)=
\frac{1}{2}\sum_{k=0}^n\binom{n}{k}_M\left(\B_{k,M}^{(\alpha)}(y;\lambda)+\lambda\B_{k,M}^{(\alpha)}(y+_M1;\lambda)\right)\E_{n-k,M}(x;\lambda).
\end{equation}
Now, by \eqref{eq:apostol-bm-difference} we have
\[
\B_{k,M}^{(\alpha)}(y;\lambda)+\lambda\B_{k,M}^{(\alpha)}(y+_M1;\lambda) = 2\B_{k,M}^{(\alpha)}(y;\lambda)+M_k\B_{k-1,M}^{(\alpha-1)}(y;\lambda).
\]
Finally, substitution of the above value in \eqref{eq:bm-em1} results in the desired expression.
\end{proof}

\begin{theorem}\label{thm:mixed-euler}
If $n \geq 0,$ then
\[
\E_{n,M}^{(\alpha)}(x+_M y;\lambda)
= \sum_{k=0}^n \frac{2}{M_{k+1}}\binom{n}{k}_M
\Bigl[\E_{k+1,M}^{(\alpha-1)}(y;\lambda) - \E_{k+1,M}^{(\alpha)}(y;\lambda)\Bigr]
\B_{n-k,M}(x;\lambda),
\]
where $\B_{n-k,M}(x;\lambda)=\B_{n-k,M}^{(1)}(x;\lambda)$ and $\E_{k+1,M}^{(\alpha-1)}(y;\lambda)=0$ when $\alpha=0.$
\end{theorem}

\begin{proof}
Consider the generating function for the Apostol-Mersenne–Euler polynomials of order $\alpha$ in variable $y$
\[
G_\alpha(t) = \left(\frac{2}{\lambda e_M^t+1}\right)^\alpha e_M^{yt}
= \sum_{n=0}^\infty \E_{n,M}^{(\alpha)}(y;\lambda) \frac{t^n}{M_n!}. 
\]
We note that
\[
G_{\alpha-1}(t) = G_{\alpha}(t)\left(\frac{\lambda e_M^{t}+1}{2}\right).
\]
Thus,
\begin{equation}\label{eq:em_g}
G_{\alpha-1}(t)-G_{\alpha}(t) = G_{\alpha}(t)\left(\frac{\lambda e_M^{t}-1}{2}\right).
\end{equation}
Let us define the following series
\begin{align*}
H_{\alpha}(t) &= \frac{2}{t}\bigl( G_{\alpha-1}(t) - G_\alpha(t) \bigr)\\
&= \sum_{k=0}^\infty \frac{2}{M_{k+1}}\bigl( \E_{k+1,M}^{(\alpha-1)}(y;\lambda) - \E_{k+1,M}^{(\alpha)}(y;\lambda) \bigr) \frac{t^k}{M_k!}.
\end{align*}
Now, multiply $H_{\alpha}(t)$ with the generating function of the Apostol–Mersenne–Bernoulli polynomials with $\alpha = 1$ and using \eqref{eq:em_g} gives
\begin{align*}
\sum_{n=0}^{\infty}\sum_{k=0}^{\infty}\binom{n}{k}_M \frac{2}{M_{k+1}}&\left(\E_{k+1,M}^{(\alpha-1)}(y;\lambda) - \E_{k+1,M}^{(\alpha)}(y;\lambda)\right) \frac{t^n}{M_n!} \\
&= G_{\alpha}(t)\left(\frac{\lambda e_M^{t}-1}{2t}\right)\left(\frac{2t}{\lambda e_M^{t}-1}\right)e_M^{xt} \\
&= \sum_{n=0}^{\infty} \E_{n,M}^{(\alpha)}(x+_M y;\lambda) \frac{t^n}{M_n!}.
\end{align*}
Comparing the corresponding coefficients yields the desired identity.
\end{proof}
Theorems \ref{thm:mixed-bernoulli} and \ref{thm:mixed-euler} provide mixed addition formulas that connect the Apostol–Bernoulli–Mersenne and Apostol–Euler–Mersenne polynomials of different orders.
The following theorem shows that the $M$-derivative acts as a lowering operator for the Apostol-type polynomials. This will be used to define the $M$-integral and to compute weighted integrals.

\begin{theorem}\label{thm:m-derivative}
For $n\ge 1$,
\begin{equation}\label{eq:apostol-bm-derivative}
\mathcal{D}^x\left(\B_{n,M}^{(\alpha)}(x;\lambda)\right) = M_n \B_{n-1,M}^{(\alpha)}(x;\lambda),
\end{equation}
\begin{equation}\label{eq:apostol-em-derivative}
\mathcal{D}^x\left(\E_{n,M}^{(\alpha)}(x;\lambda)\right) = M_n \E_{n-1,M}^{(\alpha)}(x;\lambda).
\end{equation}
\end{theorem}

\begin{proof}
Apply $\mathcal{D}^x$ to the generating function \eqref{eq:apostol-bm-def} to obtain
\begin{align*}
\mathcal{D}^x\left(\sum_{n=0}^\infty \B_{n,M}^{(\alpha)}(x;\lambda)\frac{t^n}{M_n!}\right) &= \left(\frac{t}{\lambda e_M^t-1}\right)^\alpha \mathcal{D}^x\left(e_M^{xt}\right) \\
&= \left(\frac{t}{\lambda e_M^t-1}\right)^\alpha t e_M^{xt} \\
&= t\sum_{n=0}^\infty \B_{n,M}^{(\alpha)}(x;\lambda)\frac{t^n}{M_n!} \\
&= \sum_{n=0}^\infty \B_{n,M}^{(\alpha)}(x;\lambda)\frac{t^{n+1}}{M_n!} \\
&= \sum_{n=1}^\infty M_n \B_{n-1,M}^{(\alpha)}(x;\lambda)\frac{t^n}{M_n!}.
\end{align*}
Comparing coefficients yields \eqref{eq:apostol-bm-derivative}. The proof for \eqref{eq:apostol-em-derivative} is analogous.
\end{proof}

As an application of the $M$-derivative we find the following convolution-like results.
\begin{theorem}
Let $b_n = \B_n^{(1)}(\lambda)$, then
\[
\sum_{k=0}^{n}\binom{n}{k}_M 2^k b_k b_{n-k}  = -2\bigl(M_{n-1}b_n+2^{n-2}M_nb_{n-1}\bigr).
\]
\end{theorem}
\begin{proof}
Let
\[
f(t)=\frac{t}{\lambda e_M^t-1}=\sum_{n=0}^\infty b_n\frac{t^n}{M_n!}
\]
which gives
\[
f(t)\bigl(\lambda e_M^t-1\bigr)=t .
\]
Applying the $M$-derivative to both sides we get
\begin{equation}\label{eq:deriv1}
\bigl(\lambda e_M^{2t}-1\bigr)\mathcal D^t\bigl(f(t)\bigr)+\lambda f(t)e_M^t=1.
\end{equation}
We also have $\lambda e_M^t = 1+\dfrac{t}{f(t)}$ and 
$\lambda e_M^{2t}-1 = \dfrac{2t}{f(2t)}.$ Substituting this into \eqref{eq:deriv1} gives
\[
\frac{2t}{f(2t)}\,\mathcal D^t\bigl(f(t)\bigr)+f(t)+t=1,
\]
hence
\begin{equation}\label{eq:deriv2}
\mathcal D^t\bigl(f(t)\bigr)=\frac{f(2t)}{2t}\bigl(1-t-f(t)\bigr).
\end{equation}
From this,
\[
f(2t)f(t) = f(2t)-tf(2t)-2t\mathcal{D}^t(f(t)).
\]
Now, we have 
\[
f(2t)f(t) = \sum_{n=0}^{\infty}\left(\sum_{k=0}^{n}\binom{n}{k}_M 2^k b_k b_{n-k}\right)\frac{t^n}{M_n!}
\]
and
\begin{align*}
f(2t)-tf(2t)-2t\mathcal{D}^t(f(t)) &= 1+\sum_{n=1}^{\infty}\left(2^nb_n-2^{n-1}M_{n}b_{n-1}-2M_nb_n\right)\frac{t^n}{M_n!}\\
&= 1 + \sum_{n=1}^{\infty}-2\left(M_{n-1}b_n+2^{n-2}M_nb_{n-1}\right)\frac{t^n}{M_n!}.
\end{align*}
Comparing the corresponding coefficients gives the required relation.
\end{proof}

\begin{theorem}
If $e_n = \E_n^{(1)}(\lambda)$, then
\[
\sum_{k=0}^{n}\binom{n}{k}_M2^ke_ke_{n-k} = 2\bigl(2^ne_n+M_{n+1}e_{n-1}\bigr).
\]
\end{theorem}
\begin{proof}
Define
\[
g(t)=\frac{2}{\lambda e_M^t+1}=\sum_{n=0}^\infty e_n\frac{t^n}{M_n!}.
\]
By applying the $M$-derivative we get
\[
(\lambda e_M^{2t}+1)\mathcal D^t(g(t)) + g(t)\,\mathcal D^t(\lambda e_M^t+1)=0.
\]
Since $\lambda e_M^{2t}+1=\dfrac{2}{g(2t)}$, $\mathcal D^t(\lambda e_M^t+1)=\lambda e_M^t$ and $\lambda e_M^t=\dfrac{2-g(t)}{g(t)}$, we obtain
\begin{equation}\label{eq:Eulfuneq}
g(t)g(2t) = 2\bigl(g(2t)+ \mathcal{D}^t(g(t))\bigr).
\end{equation}
Now,
\[
g(t)g(2t) = \sum_{n=0}^{\infty}\left(\sum_{k=0}^{n}\binom{n}{k}_M2^ke_ke_{n-k}\right)\frac{t^n}{M_n!},
\]
whereas
\begin{align*}
2\bigl(g(2t)+ \mathcal{D}^t(g(t))\bigr) &= 2\left(\sum_{n=0}^{\infty}2^ne_n\frac{t^n}{M_n!} + \sum_{n=1}^{\infty}M_ne_{n-1}\frac{t^n}{M_n!}\right)\\
&= \sum_{n=0}^{\infty}2\left(2^ne_n+M_{n+1}e_{n-1}\right)\frac{t^n}{M_n!}.
\end{align*}
Hence, by \eqref{eq:Eulfuneq} we prove the result.
\end{proof}

In the next result we address a closed form for weighted integrals of these Apostol-type polynomials.

\begin{theorem}\label{thm:int-bernoulli-recurrence}
For $m,n\ge 1$ and any admissible $\lambda$ and $\alpha$, define
\[
I_{m,n}^{(\alpha)}:=\int_0^1 x^n\,\B_{m,M}^{(\alpha)}(x;\lambda)\,d_M x.
\]
Then
\begin{equation}\label{eq:bernoulli-recurrence}
\begin{aligned}
I_{m,n}^{(\alpha)} = &\Bigl(\frac{1}{\lambda}-1\Bigr)\frac{\B_{m+1,M}^{(\alpha)}(\lambda)}{M_{m+1}}
      +\frac{1}{\lambda}\,\B_{m,M}^{(\alpha-1)}(\lambda) \\
     &- \frac{M_n}{M_{m+1}}\sum_{k=0}^{m}\binom{m+1}{k}_M\,
        \B_{k,M}^{(\alpha)}(\lambda)\,\frac{2^{\,m+1-k}}{M_{n+m-k+1}}.
\end{aligned}
\end{equation}
\end{theorem}

\begin{proof}
Using Theorem \ref{thm:m-derivative} we can write
\[
I_{m,n}^{(\alpha)}=\frac{1}{M_{m+1}}\int_0^1 x^n\,\mathcal{D}^x\B_{m+1,M}^{(\alpha)}(x;\lambda)\,d_M x.
\]

By applying the integration‑by‑parts formula of the $M$-integral we have
\[
\int_0^1 x^n\,\mathcal{D}^x\B_{m+1,M}^{(\alpha)}(x;\lambda)\,d_M x
   = \Bigl[x^n\B_{m+1,M}^{(\alpha)}(x;\lambda)\Bigr]_0^1
     - \int_0^1 \mathcal{D}^x(x^n)\,\B_{m+1,M}^{(\alpha)}(2x;\lambda)\,d_M x.
\]

For $n\ge 1$, $\mathcal{D}^x(x^n)=M_n x^{n-1}$ and the boundary term at $x=0$ vanishes. Thus
\begin{equation} \label{eq:bm_int}
I_{m,n}^{(\alpha)}=\frac{1}{M_{m+1}}\Bigl(
   \B_{m+1,M}^{(\alpha)}(1;\lambda)
   -M_n\int_0^1 x^{n-1}\B_{m+1,M}^{(\alpha)}(2x;\lambda)\,d_M x\Bigr).
\end{equation}

Since
\[
\B_{m+1,M}^{(\alpha)}(2x;\lambda)=\sum_{k=0}^{m+1}\binom{m+1}{k}_M\,
   \B_{k,M}^{(\alpha)}(\lambda)\,2^{\,m+1-k}x^{m+1-k},
\]
we obtain
\begin{align*}
\int_0^1 x^{n-1}\B_{m+1,M}^{(\alpha)}(2x;\lambda)\,d_M x
   &=\sum_{k=0}^{m+1}\binom{m+1}{k}_M\,
     \B_{k,M}^{(\alpha)}(\lambda)\,\frac{2^{\,m+1-k}}{M_{n+m-k+1}}\\
   &=\sum_{k=0}^{m}\binom{m+1}{k}_M\,
   \B_{k,M}^{(\alpha)}(\lambda)\,\frac{2^{\,m+1-k}}{M_{n+m-k+1}} + \frac{\B_{m+1,M}^{(\alpha)}(\lambda)}{M_n}.
\end{align*}
Substituting this into \eqref{eq:bm_int} yields
\begin{align}\label{eq:bm_int1}
I_{m,n}^{(\alpha)} &=\frac{1}{M_{m+1}}\Bigl(
   \B_{m+1,M}^{(\alpha)}(1;\lambda)
   -M_n\Bigl(\sum_{k=0}^{m}\binom{m+1}{k}_M\,
      \B_{k,M}^{(\alpha)}(\lambda)\frac{2^{\,m+1-k}}{M_{n+m-k+1}}
      +\frac{\B_{m+1,M}^{(\alpha)}(\lambda)}{M_n}\Bigr)\Bigr)\notag\\
&=\frac{1}{M_{m+1}}\Bigl(
   \B_{m+1,M}^{(\alpha)}(1;\lambda)-\B_{m+1,M}^{(\alpha)}(\lambda)
   -M_n\sum_{k=0}^{m}\binom{m+1}{k}_M\,
      \B_{k,M}^{(\alpha)}(\lambda)\frac{2^{\,m+1-k}}{M_{n+m-k+1}}\Bigr).
\end{align}
By Theorem \ref{thm:difference} we have
\[
\frac{1}{M_{m+1}}\bigl(\B_{m+1,M}^{(\alpha)}(1;\lambda)-\B_{m+1,M}^{(\alpha)}(\lambda)\bigr)
   =\Bigl(\frac{1}{\lambda}-1\Bigr)\frac{\B_{m+1,M}^{(\alpha)}(\lambda)}{M_{m+1}}
     +\frac{1}{\lambda}\B_{m,M}^{(\alpha-1)}(\lambda).
\]
Finally, substituting this into \eqref{eq:bm_int1} gives precisely \eqref{eq:bernoulli-recurrence}.  
\end{proof}

An analogous result holds for the Euler polynomials.
\begin{theorem}\label{thm:int-euler-recurrence}
For $m,n\ge 1$ and any admissible $\lambda$ and $\alpha$, define
\[
J_{m,n}^{(\alpha)}:=\int_0^1 x^n\,\E_{m,M}^{(\alpha)}(x;\lambda)\,d_M x.
\]
Then
\begin{equation}\label{eq:euler-recurrence}
\begin{aligned}
J_{m,n}^{(\alpha)} = &-\Bigl(\frac{1}{\lambda}+1\Bigr)\frac{\E_{m+1,M}^{(\alpha)}(\lambda)}{M_{m+1}}
      +\frac{2}{\lambda}\,\frac{\E_{m+1,M}^{(\alpha-1)}(\lambda)}{M_{m+1}} \\
     &\quad- \frac{M_n}{M_{m+1}}\sum_{k=0}^{m}\binom{m+1}{k}_M\,
        \E_{k,M}^{(\alpha)}(\lambda)\,\frac{2^{\,m+1-k}}{M_{n+m-k+1}}.
\end{aligned}
\end{equation}
\end{theorem}

\begin{proof}
Following the same pattern as for the Bernoulli case gives
\[
J_{m,n}^{(\alpha)}=\frac{1}{M_{m+1}}\Bigl(
   \E_{m+1,M}^{(\alpha)}(1;\lambda)
   -M_n\int_0^1 x^{n-1}\E_{m+1,M}^{(\alpha)}(2x;\lambda)\,d_M x\Bigr).
\]
Further, by expanding $\E_{m+1,M}^{(\alpha)}(2x;\lambda)$ we get
\begin{align*}
\int_0^1 x^{n-1}\E_{m+1,M}^{(\alpha)}(2x;\lambda)\,d_M x
   &= \sum_{k=0}^{m}\binom{m+1}{k}_M\,
      \E_{k,M}^{(\alpha)}(\lambda)\,\frac{2^{\,m+1-k}}{M_{n+m-k+1}}
     + \frac{\E_{m+1,M}^{(\alpha)}(\lambda)}{M_n}.
\end{align*}
Then
\begin{align*}
J_{m,n}^{(\alpha)} &=\frac{1}{M_{m+1}}\Bigl(
   \E_{m+1,M}^{(\alpha)}(1;\lambda)
   -M_n\Bigl(\sum_{k=0}^{m}\binom{m+1}{k}_M\,
      \E_{k,M}^{(\alpha)}(\lambda)\frac{2^{\,m+1-k}}{M_{n+m-k+1}}
      +\frac{\E_{m+1,M}^{(\alpha)}(\lambda)}{M_n}\Bigr)\Bigr)\\
&=\frac{1}{M_{m+1}}\Bigl(
   \E_{m+1,M}^{(\alpha)}(1;\lambda)-\E_{m+1,M}^{(\alpha)}(\lambda)
   -M_n\sum_{k=0}^{m}\binom{m+1}{k}_M\,
      \E_{k,M}^{(\alpha)}(\lambda)\frac{2^{\,m+1-k}}{M_{n+m-k+1}}\Bigr).
\end{align*}
Moreover, using the difference equation we have
\[
\frac{1}{M_{m+1}}\bigl(\E_{m+1,M}^{(\alpha)}(1;\lambda)-\E_{m+1,M}^{(\alpha)}(\lambda)\bigr)
   = -\Bigl(\frac{1}{\lambda}+1\Bigr)\frac{\E_{m+1,M}^{(\alpha)}(\lambda)}{M_{m+1}}
     +\frac{2}{\lambda}\,\frac{\E_{m+1,M}^{(\alpha-1)}(\lambda)}{M_{m+1}}.
\]
Substitution of this into the previous expression proves the theorem.  
\end{proof}
These formula can be used to compute moments and establish orthogonality-like relations.

\section{Conclusion}

In this paper, we have introduced and systematically studied the Apostol-type Mersenne-Bernoulli and Mersenne-Euler polynomials of order $\alpha$. Building on the $M$-calculus framework, we have established a comprehensive set of properties for these polynomials, including explicit representations, addition theorems, difference equations $M$-derivative formulas, $M$-integral formulas, and convolution identities. In future, the zeros and graphical representations of these polynomials could be studied computationally, which might reveal interesting patterns. One could also explore the connections between these polynomials and special functions, such as $M$-analogues of different zeta functions. Finally, the methodology developed here could be applied to other sequences, such as Pell numbers, Jacobsthal numbers, or Perrin numbers.

\section*{Conflict of interest}
The authors declare that they have no conflict of interest.

\end{document}